\documentclass[11pt,letterpaper]{article}

\usepackage[margin=1in]{geometry}
\usepackage{amsmath,amssymb,amsthm,mathtools}
\usepackage{microtype}
\usepackage[hidelinks]{hyperref}
\numberwithin{equation}{section}

\newtheorem{theorem}{Theorem}[section]
\newtheorem{proposition}[theorem]{Proposition}
\newtheorem{lemma}[theorem]{Lemma}
\newtheorem{question}[theorem]{Question}
\theoremstyle{remark}

\newcommand{\R}{\mathbb R}
\newcommand{\Sd}{S^{d-1}}
\newcommand{\op}{\mathrm{op}}
\newcommand{\Isom}{\operatorname{Isom}}
\newcommand{\dis}{\operatorname{dis}}

\title{The Uniform Gromov--Hausdorff Gap Problem for Approximating\\
Spheres by Finite Homogeneous Spaces\\[0.8em]
\large\itshape Into the Abyss with Open Eyes}
\author{Itai Benjamini}
\date{August 2026}

\begin{document}

\maketitle

\begin{abstract}
Let $S^n$ be the unit round sphere with its intrinsic angular metric, normalized so that
$\operatorname{diam}S^n=\pi$. For finite homogeneous metric spaces $X$, put
\[
  \delta_n=\inf_X d_{GH}(X,S^n).
\]
The main open problem is whether $\inf_{n\ge2}\delta_n>0$. Gelander's theorem gives
$\delta_n>0$ in each fixed dimension, but not uniformly. An abstract cross-polytope
construction gives the universal upper bound $\delta_n\le\pi/4$. In the opposite direction,
ChatGPT combines the passage from small Gromov--Hausdorff error to an approximate finite
action on the sphere, logarithmic stability of approximate inner-product-preserving maps due
to Cuesta, operator-norm stability of almost representations, and Green's width theorem for
finite transitive sets. This gives the quantitative bound
\[
  \delta_n\ge \frac{c}{(1+\log(n+1))^2}
\]
for all sufficiently large $n$. The remaining logarithm is closely related to the known gap in
the optimal stability theory for almost symmetries of Hilbert space.
\end{abstract}

\section{The problem}

A finite metric space $X$ is \emph{homogeneous} if its full isometry group acts transitively
on its points. We study
\[
  \delta_n:=\inf\{d_{GH}(X,S^n):X\text{ is finite and homogeneous}\}.
\]
We use the correspondence formula
\[
  d_{GH}(Y,Z)=\frac12\inf_R\dis(R),
\]
where $R\subseteq Y\times Z$ ranges over correspondences and
\[
  \dis(R)=\sup\{|d_Y(y,y')-d_Z(z,z')|:(y,z),(y',z')\in R\}.
\]
The circle is exceptional: regular polygons show that $\delta_1=0$.

\begin{theorem}[Gelander \cite{Gelander}]
For each $n\ge2$, one has $\delta_n>0$. More generally, a compact manifold which is a
Gromov--Hausdorff limit of finite homogeneous metric spaces is a torus.
\end{theorem}

Gelander's theorem is fixed-dimensional. The central question is its dimension-free form.

\begin{question}[Uniform gap problem]
Does there exist $\varepsilon_0>0$ such that
\[
  d_{GH}(X,S^n)\ge\varepsilon_0
\]
for every $n\ge2$ and every finite homogeneous metric space $X$?
\end{question}

Finite homogeneous spaces in this question are abstract: they need not be subsets of $S^n$.
This freedom makes the problem substantially different from approximation by finite
orthogonal orbits.

The proof below follows a four-step reduction. A small-distortion correspondence first turns
the isometry group of $X$ into an approximate action on the sphere. A stability theorem of
Cuesta replaces each approximate symmetry by an orthogonal map. Kazhdan's stability theorem
then corrects these maps simultaneously to a genuine orthogonal representation. Transitivity
of $X$ makes one orbit of that representation dense, while Green's width theorem prevents a
finite orthogonal orbit from being too dense.

\section{The cross-polytope upper bound}

Fix a finite antipodal $\eta$-net
\[
  P=\{\pm p_1,\ldots,\pm p_N\}\subset S^n,
  \qquad \eta\le\pi/4,
\]
with the pairs $\{\pm p_i\}$ distinct. Let $C_N=\{\pm1,\ldots,\pm N\}$ and give it the
metric
\[
  d_{C_N}(\varepsilon i,\varepsilon'j)=
  \begin{cases}
    0, & i=j,\ \varepsilon=\varepsilon',\\
    \pi, & i=j,\ \varepsilon=-\varepsilon',\\
    \pi/2, & i\ne j.
  \end{cases}
\]
This is the angular metric on the vertices of a cross-polytope, and signed permutations act
transitively on it.

\begin{proposition}
For every $n\ge1$,
\[
  \delta_n\le\pi/4.
\]
\end{proposition}

\begin{proof}
Choose a map $q:S^n\to C_N$ such that $d_S(y,\varepsilon p_i)\le\eta$ whenever
$q(y)=\varepsilon i$, and require $q(\varepsilon p_i)=\varepsilon i$. Thus $q$ is
surjective, and
\[
  R=\{(q(y),y):y\in S^n\}
\]
is a correspondence. If two points have the same label, their distance is at most $2\eta$;
if their labels are antipodal, their distance differs from $\pi$ by at most $2\eta$; and for
two other distinct labels the discrepancy from $\pi/2$ is at most $\pi/2$. Hence
\[
  \dis(R)\le\max\{2\eta,\pi/2\}=\pi/2,
\]
so $d_{GH}(C_N,S^n)\le\pi/4$.
\end{proof}

\begin{proposition}
If $Y$ is connected and compact and $N\ge2$, then
\[
  d_{GH}(C_N,Y)\ge\pi/4.
\]
\end{proposition}

\begin{proof}
Suppose a correspondence $R\subset C_N\times Y$ has distortion $\alpha<\pi/2$. For
$c\in C_N$, set
\[
  F_c=\{y\in Y:(c,y)\in R\}.
\]
The nonempty sets $F_c$ cover $Y$. If $c\ne c'$, then every $y\in F_c$ and
$y'\in F_{c'}$ satisfy
\[
  d_Y(y,y')\ge d_{C_N}(c,c')-\alpha\ge\pi/2-\alpha>0.
\]
Thus the fibers are pairwise disjoint and uniformly separated. Each $F_c$ is consequently
open in $Y$, and its complement, being the union of the other fibers, is open as well. They
form a finite nontrivial clopen partition of $Y$, contrary to connectedness.
\end{proof}

This proves sharpness only within the cross-polytope family, not among all finite homogeneous
metrics.

\section{The embedded benchmark}

Put $d=n+1$ and define
\[
  \beta_d=\inf_{\substack{G\le O(d)\text{ finite}\\u\in\Sd}}
  d_H(G u,\Sd),
\]
where $d_H$ is intrinsic spherical Hausdorff distance. Since $Gu\subseteq\Sd$, this is the
least possible covering radius of a finite orthogonal orbit.

\begin{theorem}[Green \cite{Green}]
There are universal constants $C>0$ and $d_0$ such that, for $d\ge d_0$,
\[
  \arccos\!\left(\frac{C}{\sqrt{\log d}}\right)
  \le\beta_d\le
  \arccos\!\left(\frac1{\sqrt{H_d}}\right),
  \qquad H_d=\sum_{k=1}^d\frac1k.
\]
In particular,
\[
  \beta_d=\frac\pi2-\Theta((\log d)^{-1/2}).
\]
\end{theorem}

\begin{proof}
Green's width theorem gives, for every finite $G\le O(d)$ and $u\in\Sd$, a unit vector
$w$ such that
\[
  \sup_{g\in G}|\langle gu,w\rangle|\le \frac{C}{\sqrt{\log d}}.
\]
Therefore $w$ is at angular distance at least
$\arccos(C/\sqrt{\log d})$ from the whole orbit, proving the lower bound.

For the upper bound, use Green's signed-permutation example: take the orbit of
\[
  a=\frac1{\sqrt{H_d}}
  \left(1,\frac1{\sqrt2},\ldots,\frac1{\sqrt d}\right)
\]
under signed coordinate permutations. Given $u\in\Sd$, let
$b_1\ge\cdots\ge b_d\ge0$ be its absolute coordinates in decreasing order. A suitable
signed permutation of $a$ has inner product
\[
  \frac1{\sqrt{H_d}}\sum_{i=1}^d\frac{b_i}{\sqrt i}
\]
with $u$. Put $b_{d+1}=0$. Since
\[
  b=\sum_{k=1}^d(b_k-b_{k+1})\mathbf 1_{[k]}
\]
and $\sum_{i=1}^k i^{-1/2}\ge\sqrt k=\|\mathbf 1_{[k]}\|_2$, the triangle inequality gives
\[
  \sum_{i=1}^d\frac{b_i}{\sqrt i}
  \ge\sum_{k=1}^d(b_k-b_{k+1})\sqrt k
  \ge\|b\|_2=1.
\]
Thus every $u$ lies within angular distance
$\arccos(1/\sqrt{H_d})$ of the orbit.
\end{proof}

The embedded obstruction tends to $\pi/2$, whereas the abstract cross-polytope stays at
distance $\pi/4$.

\section{Logarithmic stability on the sphere}

Write $B_2^d=\{x\in\R^d:\|x\|_2\le1\}$. We use the following real-Hilbert-space form
of Cuesta's stability theorem for global almost symmetries. Cuesta states the result for a
finite-dimensional Hilbert--Schmidt space. In the proof of \cite[Proposition 7]{Cuesta}, the
only properties of that space used are its real Hilbert structure and its real dimension;
running the same argument for a real Hilbert space of dimension $d$ gives the statement
below. We use $1+\log d$ to make the estimate uniform in small dimensions.

\begin{theorem}[Cuesta \cite{Cuesta}]
There is a universal $C$ such that the following holds for $d\ge2$. If a continuous map
$F:B_2^d\to B_2^d$ satisfies
\[
  |\langle F(x),F(y)\rangle-\langle x,y\rangle|\le\eta
  \qquad(x,y\in B_2^d),
\]
then there is a linear map $A:\R^d\to\R^d$ such that
\[
  \sup_{x\in B_2^d}\|F(x)-Ax\|_2
  \le C\sqrt\eta\,(1+\log d).
\]
\end{theorem}

We record the spherical form needed below.

\begin{lemma}[Logarithmic stability of rough sphere isometries]
Let $d\ge2$. There are universal constants $c,C>0$ such that if $f:\Sd\to\Sd$ satisfies
\[
  |d_S(f(u),f(v))-d_S(u,v)|\le\eta
  \qquad(u,v\in\Sd),
\]
and $\sqrt\eta(1+\log d)\le c$, then some $U\in O(d)$ satisfies
\[
  \sup_{u\in\Sd}\|f(u)-Uu\|_2
  \le C\sqrt\eta\,(1+\log d).
\]
No continuity of $f$ is assumed.
\end{lemma}

\begin{proof}
The case $\eta=0$ follows directly from exact preservation of inner products, so assume
$\eta>0$. Since cosine is $1$-Lipschitz on $[0,\pi]$,
\begin{equation}
  |\langle f(u),f(v)\rangle-\langle u,v\rangle|\le\eta.
  \label{eq:inner-product}
\end{equation}

We first replace $f$ by a continuous map at $O(\eta)$ additional error. Choose a
triangulation of $\Sd$ of angular mesh at most $\eta^2$. On each simplex, extend the values
of $f$ at its vertices by barycentric interpolation in $\R^d$. If $x_i,x_j$ are vertices of
one simplex, then
\[
  d_S(f(x_i),f(x_j))\le d_S(x_i,x_j)+\eta\le2\eta,
\]
so the images have chordal diameter at most $2\eta$. Moreover,
$\langle f(x_i),f(x_j)\rangle\ge1-C\eta^2$. Hence every barycentric combination $h$ of
these images satisfies $\|h\|_2\ge1-C\eta^2$. We may therefore normalize the interpolation
to obtain a continuous $g:\Sd\to\Sd$.

For $u$ in a simplex and any vertex $x_i$ of that simplex, both $g(u)$ and $f(u)$ are
$O(\eta)$-close to $f(x_i)$. Consequently
\begin{equation}
  \sup_{u\in\Sd}\|g(u)-f(u)\|_2\le C\eta.
  \label{eq:smoothing}
\end{equation}
It follows from \eqref{eq:inner-product}--\eqref{eq:smoothing} that
\[
  |\langle g(u),g(v)\rangle-\langle u,v\rangle|\le C\eta
  \qquad(u,v\in\Sd).
\]

Extend $g$ radially to the unit ball by $F(ru)=rg(u)$ for $0\le r\le1$. Then
\[
  |\langle F(x),F(y)\rangle-\langle x,y\rangle|\le C\eta
  \qquad(x,y\in B_2^d).
\]
Cuesta's theorem gives a linear map $A$ with
\[
  q:=\sup_{u\in\Sd}\|g(u)-Au\|_2
  \le C\sqrt\eta\,(1+\log d).
\]
Since $\|g(u)\|_2=1$, we have $1-q\le\|Au\|_2\le1+q$ for every unit vector $u$.
Thus every singular value of $A$ lies in $[1-q,1+q]$. If $q<1/2$, the polar factor
$U$ of $A$ is orthogonal and
\[
  \|A-U\|_\op\le q.
\]
Combining this with \eqref{eq:smoothing} proves the assertion after adjusting the universal
constant.
\end{proof}

\section{From a homogeneous approximation to an almost action}

We also use operator-norm Ulam stability for finite-group representations. Kazhdan's
proof is usually stated for unitary operators. Its averaging and polar-decomposition steps
are valid over a real Hilbert space and give the orthogonal form below.

\begin{theorem}[Kazhdan \cite{Kazhdan}]
There are universal constants $\alpha_0,C>0$ with the following property. Let $\Gamma$ be
finite. If $V:\Gamma\to O(d)$, $V_e=I$, and
\[
  \|V_\gamma V_\delta-V_{\gamma\delta}\|_\op
  \le\alpha\le\alpha_0,
\]
then there is an orthogonal representation $\rho:\Gamma\to O(d)$ satisfying
\[
  \|V_\gamma-\rho(\gamma)\|_\op\le C\alpha.
\]
\end{theorem}

Let $S=\Sd$ and suppose that $X$ is finite homogeneous with
$d_{GH}(X,S)<\varepsilon$. Choose a correspondence $R\subset X\times S$ with
distortion $D<2\varepsilon$. Write
\[
  F_x=\{u\in S:(x,u)\in R\}.
\]
Choose anchors $a_x\in F_x$ and a selector $s:S\to X$ with $u\in F_{s(u)}$. Fix
$x_0\in X$, set $u_0=a_{x_0}$, and arrange $s(u_0)=x_0$.

Let $\Gamma=\Isom(X)$, which is finite. Define
\[
  f_\gamma(u)=a_{\gamma s(u)}.
\]
For $u,v\in S$, the correspondence estimate applied before and after the isometry $\gamma$
gives
\begin{align*}
  |d_S(f_\gamma(u),f_\gamma(v))-d_S(u,v)|
  &\le |d_S(a_{\gamma s(u)},a_{\gamma s(v)})
          -d_X(\gamma s(u),\gamma s(v))|\\
  &\quad+|d_X(s(u),s(v))-d_S(u,v)|\\
  &\le2D.
\end{align*}
Thus
\begin{equation}
  |d_S(f_\gamma(u),f_\gamma(v))-d_S(u,v)|\le2D.
  \label{eq:rough-isometry}
\end{equation}

Furthermore, $f_\delta(u)$ lies both in $F_{\delta s(u)}$ and in
$F_{s(f_\delta(u))}$. Hence
\[
  d_X(\delta s(u),s(f_\delta(u)))\le D.
\]
After applying $\gamma$ and comparing the two anchors through the correspondence, we obtain
\begin{equation}
  d_S(f_\gamma(f_\delta(u)),f_{\gamma\delta}(u))\le2D.
  \label{eq:almost-action}
\end{equation}

\section{The logarithmic-squared lower bound}

\begin{theorem}
There is a universal constant $c>0$ such that, for all sufficiently large $d$,
\[
  \inf_{X\text{ finite homogeneous}}d_{GH}(X,\Sd)
  \ge \frac{c}{(1+\log d)^2}\min\{1,\beta_d^2\}.
\]
Consequently, for all sufficiently large $n$,
\[
  \delta_n\ge\frac{c}{(1+\log(n+1))^2}.
\]
\end{theorem}

\begin{proof}
Use the construction of Section 5 and put
\[
  L=1+\log d,
  \qquad q=C_0\sqrt D\,L,
\]
where $C_0$ is a sufficiently large universal constant. If $q$ exceeds a sufficiently small
universal threshold, then $D\ge cL^{-2}$, which implies the asserted bound after changing
$c$. We may therefore assume that $q$ is small enough for all the stability statements below.

Lemma 4.2, applied to \eqref{eq:rough-isometry}, gives $U_\gamma\in O(d)$ such that
\begin{equation}
  \sup_{u\in S}\|f_\gamma(u)-U_\gamma u\|_2\le q.
  \label{eq:linearization}
\end{equation}
For the identity element, $f_e(u)=a_{s(u)}$ and $u$ belong to the same fiber, so
$d_S(f_e(u),u)\le D$. Hence
\[
  \|U_e-I\|_\op\le q+D\le Cq.
\]

Before normalizing at the identity, \eqref{eq:almost-action} and
\eqref{eq:linearization} give, for every $u\in S$,
\begin{align*}
  \|U_\gamma U_\delta u-U_{\gamma\delta}u\|_2
  &\le \|U_\delta u-f_\delta(u)\|_2
      +\|U_\gamma f_\delta(u)-f_\gamma(f_\delta(u))\|_2\\
  &\quad+\|f_\gamma(f_\delta(u))-f_{\gamma\delta}(u)\|_2
      +\|f_{\gamma\delta}(u)-U_{\gamma\delta}u\|_2\\
  &\le3q+2D\le Cq.
\end{align*}
Thus $\|U_\gamma U_\delta-U_{\gamma\delta}\|_\op\le Cq$. Set
$V_\gamma=U_\gamma U_e^{-1}$. Then $V_e=I$, and
$\sup_{u\in S}\|f_\gamma(u)-V_\gamma u\|_2\le Cq$. Moreover,
\begin{align*}
  \|V_\gamma V_\delta-V_{\gamma\delta}\|_\op
  &=\|U_\gamma U_e^{-1}U_\delta-U_{\gamma\delta}\|_\op\\
  &\le \|U_\gamma U_\delta-U_{\gamma\delta}\|_\op
      +\|U_e^{-1}-I\|_\op\\
  &\le Cq.
\end{align*}
Choose the preceding threshold so that $Cq\le\alpha_0$. Kazhdan's theorem produces an
orthogonal representation $\rho:\Gamma\to O(d)$ satisfying
\begin{equation}
  \|V_\gamma-\rho(\gamma)\|_\op\le Cq.
  \label{eq:representation}
\end{equation}

We claim that $\rho(\Gamma)u_0$ is $Cq$-dense in $S$. Given $y\in S$, choose
$x\in X$ with $y\in F_x$ and choose $\gamma\in\Gamma$ with $\gamma x_0=x$. Since
$y,a_x\in F_x$,
\[
  d_S(y,a_x)\le D,
  \qquad a_x=f_\gamma(u_0).
\]
For unit vectors $v,w$, one has
$d_S(v,w)\le(\pi/2)\|v-w\|_2$. By the normalized linearization estimate and
\eqref{eq:representation},
\[
  \|f_\gamma(u_0)-\rho(\gamma)u_0\|_2
  \le \|f_\gamma(u_0)-V_\gamma u_0\|_2
     +\|V_\gamma-\rho(\gamma)\|_\op
  \le Cq.
\]
Therefore
\[
  d_S(y,\rho(\gamma)u_0)
  \le D+Cq\le Cq.
\]
By the definition of $\beta_d$,
\[
  \beta_d\le Cq=C\sqrt D\,(1+\log d).
\]
Consequently
\[
  D\ge\frac{c\beta_d^2}{(1+\log d)^2}.
\]
Together with the large-$q$ case, this proves
\[
  D\ge\frac{c}{(1+\log d)^2}\min\{1,\beta_d^2\}.
\]
Since $D$ may be taken arbitrarily close to twice the Gromov--Hausdorff distance, the first
assertion follows. Green's theorem gives $\beta_d\ge c_0>0$ for all sufficiently large $d$,
which proves the second.
\end{proof}

\section{The remaining stability problem}

The proof isolates a concrete route toward the uniform-gap conjecture. Lemma 4.2 supplies an
error of order
\[
  \sqrt\eta\,(1+\log d)
\]
when an individual rough sphere isometry is replaced by an orthogonal map. Replacing this by
an estimate of order
\[
  \sqrt{\eta(1+\log d)}
\]
would improve the conclusion to $\delta_n\gtrsim1/\log n$.

The maps arising here are not arbitrary: they satisfy the approximate multiplication law
\eqref{eq:almost-action}, and the orbit of $u_0$ is approximately dense. It remains open
whether this collective
structure permits a stronger simultaneous correction than one obtains by stabilizing the maps
one at a time. A dimension-free simultaneous correction would imply a uniform
Gromov--Hausdorff gap.

\paragraph{Disclosure.}
The problem was suggested by the author. ChatGPT provided the proofs and wrote the paper.

\end{document}